\documentclass[10pt]{amsart}
\usepackage{amssymb}
\usepackage{epsfig}
\usepackage{url}
\usepackage[T1]{fontenc}
\usepackage{color}
\usepackage{tikz}
\usetikzlibrary{arrows}
\theoremstyle{plain}
\makeatletter
\@namedef{subjclassname@2020}{%
  \textup{2020} Mathematics Subject Classification}
\makeatother
\newtheorem{theorem}{Theorem}[section]
\newtheorem{definition}[theorem]{Definition}
\newtheorem{corollary}[theorem]{Corollary}
\newtheorem{lem}[theorem]{Lemma}
\newtheorem{prop}[theorem]{Proposition}

\def\bbb{\mathbb}

\renewcommand{\phi}{\varphi}

\newcommand{\N}{\bbb{N}}
\newcommand{\Z}{\bbb{Z}}
\newcommand{\Q}{\bbb{Q}}

\newcommand{\bs}{\backslash}

\begin{document}
\title{Binomial coefficients, roots of unity and powers of prime numbers}
\author{Piotr Miska}
\address{Institute of Mathematics\\ Jagiellonian University in Krak\'{o}w\\ Krak\'{o}w, Poland\\ ORCID: 0000-0002-6792-8634.}
\email{piotr.miska@uj.edu.pl}

\thanks{The research of the author was partially supported by the grant of the Polish National Science Centre no. UMO-2019/34/E/ST1/00094 and the scholarship START 2019 of the Foundation for Polish Science no. START 59.2019.}

\maketitle

\begin{abstract}
Let $t\in\N_+$ be given. In this article we are interested in characterizing those $d\in\N_+$ such that the congruence
$$\frac{1}{t}\sum_{s=0}^{t-1}{n+d\zeta_t^s\choose d-1}\equiv {n\choose d-1}\pmod{d}$$
is true for each $n\in\Z$. In particular, assuming that $d$ has a prime divisor greater than $t$, we show that the above congruence holds for each $n\in\Z$ if and only if $d=p^r$, where $p$ is a prime number greater than $t$ and $r\in\{1,\ldots ,t\}$.
\end{abstract}

\section{Introduction}

By $\N$, $\N_+$, $\Z$ and $\Q$ we denote the sets of non-negative integers, positive integers, integers and rational numbers, respectively. 

For a fixed prime number $p$ and a non-zero rational number $x$, the value $\nu_p(x)$ is the unique integer $t$ such that $x=p^t\cdot\frac{a}{b}$ for some integers $a,b$ not divisible by $p$. For $x=0$ we put $\nu_p(0)=+\infty$. The value $\nu_p(x)$ is called the $p$-adic valuation of the number $x$. 

In \cite{MiUl} we studied the numbers $H_d(n)$ of permutations in $S_n$ which can be written as products of pairwise disjoint cycles of length $d$. We also considered related polynomials $$H_d(n,x)=\sum_{j=0}^{\left\lfloor\frac{n}{d}\right\rfloor}\frac{n!}{(n-dj)!j!d^j}x^{n-dj},\quad d,n\in\N,\quad d\geq 2.$$ These polynomials generalize the numbers $H_d(n)$ as $H_d(n,1)=H_d(n)$. Let us note that the coefficient of the $i$-th power of $x$ of the polynomial $H_d(n,x)$ is the number of permutations in $S_n$ that are the products of pairwise disjoint $d$-cycles and having exactly $i$ fixed points. One of the last results in that paper is that the congruence
$$H_d(n+d,x)H_d(n-d,x)\equiv H_d(n,x)^2\pmod{d}$$
is true for any positive integer $n\geq d$ if and only if $d$ is a prime number or a square of some prime number. The proof of the result is reduced to the study of congruences of the form
$${n+d\choose d-1}+{n-d\choose d-1}\equiv 2{n\choose d-1}\pmod{d}.$$
We showed that if $d$ is not a power of $2$, then the above congruence holds for each positive integer $n\geq d$ if and only if $d\in\{p,p^2\}$ for some odd prime number $p$. This is our motivation to consider congruences
\begin{align}\label{1}
\frac{1}{t}\sum_{s=0}^{t-1}{n+d\zeta_t^s\choose d-1}\equiv {n\choose d-1}\pmod{d}
\end{align}
for fixed $t\in\N_+$, where $\zeta_t=e^{\frac{2\pi i}{t}}$ is the $t$-th primitive root of unity. Here we understand $x\choose k$, $k\in\N$, as $\frac{(x)_k}{k!}$, where 
$$(x)_k=\prod_{i=0}^{k-1} (x-i)$$
and we assume that a product over the empty set is equal to $1$. Moreover, for $x,y\in\Q$ and $d\in\N_+$ we write $x\equiv y\pmod{d}$ if $d$ divides the numerator of $x-y$ written as an irreducible quotient of integers (in the sequel we will see that the left hand side of \eqref{1} is indeed a rational number). In other words, if $\nu_p(x-y)\geq\nu_p(d)$ for each prime divisor $p$ of the number $d$. We observed that (\ref{1}) holds for each $n\in\Z$ if and only if $d$ is a power of a prime number with exponent at most equal to $t$. As we will see in the further part of the paper, this conjecture is true under the additional assumption that $d$ has a prime divisor greater than $t$. Actually, for every $t\in\N_+$ we will classify (up to finitely many exceptions) all the values of $d$ for which (\ref{1}) holds for each $n\in\Z$.

We note that this is not the first result which states that some congruence or family of congruences depending on $d$ holds for (almost) all prime values of $d$ (or some expression dependent on $d$) and does not hold in general for composite values of $d$. Here we have several famous classical results of this type:
\begin{itemize}
\item if $d$ is a prime number, then $a^d\equiv a\pmod{d}$ for each $a\in\Z$ (Fermat's little theorem);
\item $d$ is a prime number if and only if $(d-1)!\equiv -1\pmod{d}$ (Wilson's theorem);
\item if $d$ is a prime number greater than $3$, then ${2d-1\choose d-1}\equiv 1\pmod{d^3}$ (Wolstenholme's theorem, for its variations see \cite{ChamDil, HelTer, McIn, McInRoe, Zhao});
\item if $d$ is a prime number, then for any $m,n\in\N$ we have
$${m\choose n}\equiv\prod_{j=0}^k{m_j\choose n_j}\pmod{d},$$
where $m=\sum_{j=0}^km_jd^j$ and $n=\sum_{j=0}^kn_jd^j$ are the base $d$ expansions of $m$ and $n$, respectively (Lucas's theorem, for its generalizations see \cite{Bai, DavWeb1, Gran, Mes, Zhao});
\item if $M_d=2^d-1$, $d\geq 2$, and $(S_j)_{j\in\N}$ is given by the recurrence
$$S_0=0,\quad S_j=S_{j-1}^2-2,\quad j>0,$$
then $M_d$ is a prime number if and only if $S_{d-2}\equiv 0\pmod{M_d}$ (Lucas-Lehmer primality test for Mersenne numbers, see e.g. \cite{Bru, Rod});
\item if $d$ is a prime number, then 
$$\sum_{0\leq k\leq n+d-1, k\equiv r\pmod{d-1}}{n+d-1\choose k}\equiv\sum_{0\leq k\leq n, k\equiv r\pmod{d-1}}{n\choose k}\pmod{d}$$ 
(J. W. L. Glaisher, 1899, see \cite{Gla, Gran}, for generalization of Glaisher's congruence see \cite{SunTau1}).
\end{itemize}
As we see, the above facts have been stimulating research in number theory up to now. However, there are many more results on congruences modulo prime numbers or prime powers involving binomial coefficients, e.g. \cite{CaoSun, DavWeb2, MengSun, Sun1, Sun2, Sun3, SunDav, SunTau2, Sun4, Sun5, Wan, ZhaoSun}.

In \cite{MiUl} we checked for which values of $d$ the congruence (\ref{1}) with $t=2$ is satisfied for all $n\geq d$. However, it is easy to see that it is the same as to check (\ref{1}) for all $n\in\Z$. Indeed, if $n_1\equiv n_2\pmod{td!}$, then 
$$\frac{1}{t}\sum_{s=0}^{t-1}{n_1+d\zeta_t^s\choose d-1}\equiv \frac{1}{t}\sum_{s=0}^{t-1}{n_2+d\zeta_t^s\choose d-1}\pmod{d}$$
and 
$${n_1\choose d-1}\equiv {n_2\choose d-1}\pmod{d}.$$

\section{Main results}

Before we state the main theorem of the paper, we will define the set $A_t$ and condition $C_t$ for $t\in\N_+$.

\begin{definition}
Let $t\in\N_+$. We define the set $A_t$ as the set of all the values of $d\in\N_+$ satisfying one of the following conditions:
\begin{itemize}
\item $d=p^r$, where $r$ is an integer lying in the interval $$\left((\alpha_p+1)(t+1)-\frac{p^{\alpha_p+1}-1}{p-1},(\alpha_p+1)t-\nu_p(t!)\right],$$ where $\alpha_p=\lfloor\log_p t\rfloor$, and $p=2$ or $p$ is a prime number less than $t$ such that $t$ is not a power of $p$,
\item $d=p_1^{r_1}\cdots p_u^{r_u}$, where $u\geq 2$, $p_1,\ldots ,p_u$ are pairwise distinct prime numbers less than or equal to $t$, where $r_1,\ldots ,r_u\in\N_+$ and $n/p_l^{r_l}\leq t$ for each $l\in\{1,\ldots ,u\}$.
\end{itemize}
\end{definition}

\begin{definition}
Let $t\in\N_+$. We say that $d\in\N_+$ satisfies the condition $C_t$ if
\begin{align*}
\frac{1}{t}\sum_{s=0}^{t-1}{n+d\zeta_t^s\choose d-1}\equiv {n\choose d-1}\pmod{d}
\end{align*}
holds for each $n\in\Z$.
\end{definition}

The main result of the paper is the following.

\begin{theorem}\label{main}
Let us fix $t\in\N_+$. Then $d\in\N_+\bs A_t$ satisfies the condition $C_t$ if and only if one of the following conditions holds:
\begin{itemize}
\item $d\leq t$; in this case we may replace the symbol $\equiv$ by $=$,
\item $d=p^r$, where $p$ is a prime number and $r$ is a positive integer at most equal to $$(\alpha_p+1)(t+1)-\frac{p^{\alpha_p+1}-1}{p-1}.$$
\end{itemize}
\end{theorem}

The following result is a direct consequence of Theorem \ref{main}.

\begin{corollary}\label{simple}
Let $t,d\in\N_+$ be such that $d$ has a prime divisor greater than $t$. Then the condition $C_t$ is satisfied if and only if $d=p^r$ for some prime number $p>t$ and a positive integer $r\leq t$.
\end{corollary}

\section{Proof of Theorem \ref{main}}

It may not seem obvious that the value $\frac{1}{t}\sum_{s=0}^{t-1}{n+d\zeta_t^s\choose d-1}$ is rational. However, a straightforward computation shows the following.

\begin{lem}\label{lem1}
For each $c\in\N$ we have
\begin{align*}
\frac{1}{t}\sum_{s=0}^{t-1}{n+d\zeta_t^s\choose c}-{n\choose c}=\frac{1}{c!}\sum_{k=1}^{\left\lfloor\frac{c}{t}\right\rfloor}d^{kt}\sum_{n-c+1\leq j_1<\ldots <j_{c-kt}\leq n} j_1\cdot \ldots \cdot j_{c-kt}.
\end{align*}
\end{lem}

At this point we can begin the proof of Theorem \ref{main}. We split the proof into several lemmas and propositions. First, we give conditions equivalent to $C_t$. These conditions will be useful in the sequel.

\begin{lem}\label{equivconds}
Let $t,d$ be positive integers. Then the following conditions are equivalent:
\begin{description}
\item[(i)] $d$ satisfies the condition $C_t$;
\item[(ii)] for each $n\in\Z$ and $c\in\{0,\ldots ,d-1\}$ the congruence
\begin{align}\label{gid}
\frac{1}{t}\sum_{s=0}^{t-1}{n+d\zeta_t^s\choose c}\equiv {n\choose c}\pmod{d}
\end{align}
is satisfied;
\item[(iii)] for each $c\in\{0,\ldots ,d-1\}$ and some fixed $n_0\in\Z$ the congruence
$$\frac{1}{t}\sum_{s=0}^{t-1}{n_0+d\zeta_t^s\choose c}\equiv {n_0\choose c}\pmod{d}$$
is satisfied.
\end{description}
\end{lem}

\begin{proof}
The implications $(ii)\Rightarrow (iii)$ and $(ii)\Rightarrow (i)$ are obvious.

The implication $(i)\Rightarrow (ii)$ follows easily from the identity ${x+1\choose c+1}={x\choose c+1}+{x\choose c}$.

The proof of implication $(iii)\Rightarrow (ii)$ is performed by induction on $c\in\{0,\ldots ,d-1\}$. First, let us note that for each $n\in\Z$ we have
$$\frac{1}{t}\sum_{s=0}^{t-1}{n+d\zeta_t^s\choose 0}\equiv {n\choose 0}\pmod{d}.$$
Assume now that for some $c\in\{0,\ldots ,d-2\}$ and each $n\in\Z$ we have
$$\frac{1}{t}\sum_{s=0}^{t-1}{n+d\zeta_t^s\choose c}\equiv {n\choose c}\pmod{d}.$$
Hence, if
$$\frac{1}{t}\sum_{s=0}^{t-1}{n+d\zeta_t^s\choose c+1}\equiv {n\choose c+1}\pmod{d},$$
then
\begin{align*}
& \frac{1}{t}\sum_{s=0}^{t-1}{n+1+d\zeta_t^s\choose c+1}=\frac{1}{t}\sum_{s=0}^{t-1}\left({n+d\zeta_t^s\choose c+1}+{n+d\zeta_t^s\choose c}\right)\\
& \equiv {n\choose c+1}+{n\choose c}={n+1\choose c+1}\pmod{d}
\end{align*}
and
\begin{align*}
& \frac{1}{t}\sum_{s=0}^{t-1}{n-1+d\zeta_t^s\choose c+1}=\frac{1}{t}\sum_{s=0}^{t-1}\left({n+d\zeta_t^s\choose c+1}-{n-1+d\zeta_t^s\choose c}\right)\\
& \equiv {n\choose c+1}-{n-1\choose c}={n-1\choose c+1}\pmod{d}.
\end{align*}
Since
$$\frac{1}{t}\sum_{s=0}^{t-1}{n_0+d\zeta_t^s\choose c+1}\equiv {n_0\choose c+1}\pmod{d},$$
we thus obtain
$$\frac{1}{t}\sum_{s=0}^{t-1}{n+d\zeta_t^s\choose c+1}\equiv {n\choose c+1}\pmod{d}$$
for every integer $n$.
\end{proof}

A direct consequence of Lemma \ref{lem1} is the following

\begin{corollary}\label{equivcond}
For each $d,c\in\N$ and $n\in\Z$, where $c<d$, the congruence (\ref{gid}) is equivalent to
\begin{align}\label{cong}
\frac{1}{c!}\sum_{k=1}^{\left\lfloor\frac{c}{t}\right\rfloor}d^{kt}\sum_{n-c+1\leq j_1<\ldots <j_{c-kt}\leq n} j_1\cdot \ldots \cdot j_{c-kt}\equiv 0\pmod{d}.
\end{align}
\end{corollary}

From Corollary \ref{equivcond} we easily conclude equality in (\ref{gid}) for $c<t$.

\begin{corollary}\label{equal}
If $c<t$, then
\begin{align*}
\frac{1}{t}\sum_{s=0}^{t-1}{n+d\zeta_t^s\choose c}={n\choose c}
\end{align*}
for each $n\in\Z$. In particular, the condition $C_t$ is satisfied by each $d\in\{1,\ldots ,t\}$.
\end{corollary}

The next two propositions allow us to claim that if $d\not\in A_t$ is not a power of a prime number or is a power of a prime number with too large an exponent, then $d$ does not satisfy $C_t$. The first proposition implies that if $d$ has a prime divisor $p$ such that $\frac{d}{p^{\nu_p(d)}}>t$, then $d$ does not satisfy $C_t$.

\begin{prop}\label{beta}
Let $d$, $\beta_p$ be positive integers such that $\nu_p(d)>t\beta_p-\nu_p(t!)$ and $d>tp^{\nu_p(d)-\beta_p}$ for some prime divisor $p$ of $d$. Then $d$ does not satisfy the condition $C_t$.
\end{prop}

\begin{proof}
By Lemma \ref{equivconds} and Corollary \ref{equivcond}, it suffices to show that (\ref{cong}) does not hold for some $c<d$ and $n\in\Z$. We put $c=tp^{\nu_p(d)-\beta_p}$ and $n=-1$. Then (\ref{cong}) takes the form
\begin{equation}\label{c1}
\begin{split}
& \frac{1}{(tp^{\nu_p(d)-\beta_d})!}\sum_{k=1}^{p^{\nu_p(d)-\beta_p}}d^{kt}\\
& \sum_{-tp^{\nu_p(d)-\beta_p}\leq j_1<\ldots <j_{tp^{\nu_p(d)-\beta_p}-kt}\leq -1} j_1\cdot \ldots \cdot j_{tp^{\nu_p(d)-\beta_p}-kt}\\
& =(-1)^{p^{\nu_p(d)-\beta_d}t}\sum_{k=1}^{p^{\nu_p(d)-\beta_p}}\sum_{-tp^{\nu_p(d)-\beta_p}\leq i_1<\ldots <i_{kt}\leq -1} \frac{d^{kt}}{i_1\cdot \ldots \cdot i_{kt}}\equiv 0\pmod{d}.
\end{split}
\end{equation}
The summand on the left hand side with the least $p$-adic valuation is equal to $$(-1)^{tp^{\nu_p(d)}}\frac{d^{t}}{\prod_{l=1}^t ((l-t-1)p^{\nu_p(d)-\beta_p})},$$ obtained for $k=1$ and $i_l=(l-t-1)p^{\nu_p(d)-\beta_p}$, where $l\in\{1,\ldots ,t\}$. Any other summand has strictly greater $p$-adic valuation because any quotient of the form $\frac{d}{i_l}$, where $-tp^{\nu_p(d)-\beta_p}\leq i_l\leq -1$ and $i_l\neq -mp^{\nu_p(d)-\beta_p}$ for any $m\in\{1,\ldots ,t\}$, has the $p$-adic valuation greater than the $p$-adic valuation of $\frac{d}{-mp^{\nu_p(d)-\beta_p}}$. Hence, the left hand side of (\ref{c1}) has the $p$-adic valuation equal to
\begin{align*}
\nu_p\left(\frac{d^{t}}{\prod_{l=1}^t ((t-l+1)p^{\nu_p(d)-\beta_p})}\right)& =t\nu_p(d)-t(\nu_p(d)-\beta_p)-\nu_p(t!)\\
& =t\beta_p-\nu_p(t!)<\nu_p(d).
\end{align*}
This means that the left hand side of (\ref{c1}) cannot be congruent to $0$ modulo $d$.
\end{proof}

In particular, we get the following.

\begin{corollary}\label{toolargevalue}
Let $d$ be a positive integer such that $d>tp^{\nu_p(d)}$ for some prime divisor $p$ of $d$. Then $d$ does not satisfy the condition $C_t$.
\end{corollary}

\begin{proof}
We take $\beta_p=0$ and we see that $\nu_p(d)>0\geq-\nu_p(t!)$. Thus we can apply the previous proposition.
\end{proof}

The next proposition states that if $\nu_p(d)$ is too large, then $C_t$ is not satisfied.

\begin{prop}\label{toolargevaluation}
Let $d$ be a positive integer such that $\nu_p(d)>(\alpha_p+1)t-\nu_p(t!)$ for some prime number $p$, where $\alpha_p=\lfloor\log_p t\rfloor$. Then $d$ does not satisfy the condition $C_t$.
\end{prop}

\begin{proof}
By Lemma \ref{equivconds} and Corollary \ref{equivcond}, it suffices to show that (\ref{cong}) does not hold for some $c<d$ and $n\in\Z$. We take $c=tp^{\nu_p(d)-\alpha_p-1}$ and $n=-1$ in (\ref{cong}) and check the validity of the following congruence:
\begin{equation}\label{c3}
\begin{split}
& \frac{1}{(tp^{\nu_p(d)-\alpha_p-1})!}\sum_{k=1}^{p^{\nu_p(d)-\alpha_p-1}}d^{kt}\\
& \sum_{-tp^{\nu_p(d)-\alpha_p-1}\leq j_1<\ldots <j_{tp^{\nu_p(d)-\alpha_p-1}-kt}\leq -1} j_1\cdot \ldots \cdot j_{tp^{\nu_p(d)-\alpha_p-1}-kt}\\
& =(-1)^{tp^{\nu_p(d)-\alpha_p-1}}\sum_{k=1}^{p^{\nu_p(d)-\alpha_p-1}}\sum_{-tp^{\nu_p(d)-\alpha_p-1}\leq i_1<\ldots <i_{kt}\leq -1} \frac{d^{kt}}{i_1\cdot \ldots \cdot i_{kt}}\\
& \equiv 0\pmod{d}.
\end{split}
\end{equation}
The summand on the left hand side with the least $p$-adic valuation is equal to $$\frac{d^t}{\prod_{l=1}^t (l-t-1)p^{\nu_p(d)-\alpha_p-1}},$$ obtained for $k=1$ and $i_l=(l-t-1)p^{\nu_p(d)-\alpha_p-1}$, where $l\in\{1,\ldots ,t\}$. Any other summand has strictly greater $p$-adic valuation because any quotient of the form $\frac{d}{i_l}$, where $-tp^{\nu_p(d)-\alpha_p-1}\leq i_l\leq -1$ and $i_l\neq -mp^{\nu_p(d)-\alpha_p-1}$ for any $m\in\{1,\ldots ,t\}$, has the $p$-adic valuation greater than the $p$-adic valuation of $\frac{d}{-mp^{\nu_p(d)-\beta_p}}$. Hence, the left hand side of (\ref{c3}) has the $p$-adic valuation equal to
\begin{align*}
\nu_p\left(\frac{d^{t}}{\prod_{l=1}^t ((t-l+1)p^{\nu_p(d)-\alpha_p-1})}\right)& =t\nu_p(d)-t(\nu_p(d)-\alpha_p-1)-\nu_p(t!)\\
& =t(\alpha_p+1)-\nu_p(t!)<\nu_p(d),
\end{align*}
which means that it cannot be congruent to $0$ modulo $d$.
\end{proof}

The next proposition concerns a particular case of the values of $t$ and $d$.

\begin{prop}\label{particular}
Let $t=qp^u$, $d=qp^r$, where $p$ is an odd prime number and $q,u,r$ are positive integers such that $p\nmid q$ and $r>ut+1-\nu_p((t-1)!)$. Then $d$ does not satisfy the condition $C_t$.
\end{prop}

\begin{proof}
By Lemma \ref{equivconds} and Corollary \ref{equivcond}, it suffices to show that (\ref{cong}) does not hold for some $c<d$ and $n\in\Z$. We take $c=(t-1)p^{r-u}+p^{r-u-1}$ and $n=-1$ in (\ref{cong}) and check the validity of the following congruence:
\begin{equation}\label{c5}
\begin{aligned}
& \frac{1}{((t-1)p^{r-u}+p^{r-u-1})!}\sum_{k=1}^{\left\lfloor\frac{(t-1)p^{r-u}+p^{r-u-1}}{t}\right\rfloor}q^{kt}p^{rkt}\\
& \sum_{-(t-1)p^{r-u}-p^{r-u-1}\leq j_1<\ldots <j_{(t-1)p^{r-u}+p^{r-u-1}-kt}\leq -1} j_1\cdot \ldots \cdot j_{(t-1)p^{r-u}+p^{r-u-1}-kt}\\
& =(-1)^{(t-1)p^{r-u}+p^{r-u-1}}\sum_{k=1}^{\left\lfloor\frac{(t-1)p^{r-u}+p^{r-u-1}}{t}\right\rfloor}\\
& \sum_{-(t-1)p^{r-u}-p^{r-u-1}\leq i_1<\ldots <i_{kt}\leq -1} \frac{q^{kt}p^{rkt}}{i_1\cdot \ldots \cdot i_{kt}}\\
& \equiv 0\pmod{qp^r}.
\end{aligned}
\end{equation}
The summands on the left hand side with the least $p$-adic valuation are of the form $$\frac{q^{t}p^{rt}}{-p^{r-u-1}a\prod_{l=1}^{t-1} (-lp^{r-u})},$$ where $a\in\{1,\ldots ,p(t-1)+1\}$ and $p\nmid a$. Their sum is equal to
\begin{equation*}
\begin{aligned}
& \sum_{a=1, p\nmid a}^{p(t-1)+1}\frac{q^{t}p^{rt}}{-p^{r-u-1}a\prod_{l=1}^{t-1} -lp^{r-u}}=\frac{(-q)^tp^{tu+1}}{(t-1)!}\sum_{a=1, p\nmid a}^{p(t-1)+1}\frac{1}{a}\\
& \equiv\frac{(-q)^tp^{tu+1}}{(t-1)!}\left(1+(t-1)\sum_{b=1}^{p-1}b\right)\equiv\frac{(-q)^tp^{tu+1}}{(t-1)!}\left(1+\frac{(t-1)(p-1)p}{2}\right)\\
& \equiv\frac{(-q)^tp^{tu+1}}{(t-1)!}\pmod{p^{tu+2-\nu_p((t-1)!)}}
\end{aligned}
\end{equation*}
and hence its $p$-adic valuation is equal to $ut+1-\nu_p((t-1)!)$. Any other summand has $p$-adic valuation greater than $ut+1-\nu_p((t-1)!)$. Thus, the left hand side of (\ref{c5}) has $p$-adic valuation $tu+1-\nu_p((t-1)!)<r$, which means that this congruence does not hold.
\end{proof}

As an immediate consequence of the last proposition, we obtain the following.

\begin{corollary}\label{powernot}
Let $t=p^u$, $d=p^r$, where $p$ is a prime number and $u,r$ are positive integers such that $r>ut+1-\nu_p((t-1)!)$. Then the condition $C_t$ does not hold.
\end{corollary}

The following proposition gives a sufficient condition for $d$ to satisfy the condition $C_t$.

\begin{prop}\label{suffcond}
Let 
$$\nu_p(d)\leq (\nu_p(d)-\gamma_p(d)+\alpha_p+1)k_p(d)t-\sum_{g=0}^{\alpha_p}\left\lfloor\frac{d-1}{p^{\gamma_p(d)-g}}\right\rfloor$$
for each prime divisor $p$ of $d$, where 
\begin{align*}
k_p(d)=&\max\left\{1,\left\lceil\frac{1}{t}\left\lfloor\frac{d-1}{p^{\nu_p(d)+1}}\right\rfloor\right\rceil\right\},\\
\alpha_p=&\left\lfloor\log_p \frac{k_p(d)tp^{\gamma_p(d)}}{d}\right\rfloor ,\\
\gamma_p(d)=&\lfloor\log_p d\rfloor.
\end{align*}
Then $d$ satisfies the condition $C_t$.
\end{prop}

\begin{proof}
By Lemma \ref{equivconds}, it suffices to check the validity of the congruences
\begin{align}\label{-1}
\frac{1}{t}\sum_{s=0}^{t-1}{-1+d\zeta_t^s\choose c}\equiv {-1\choose c}\pmod{p^{\nu_p(d)}}
\end{align}
for each $c\in\{0,\ldots ,d-1\}$ and prime number $p$ dividing $d$. From Corollary \ref{equal} we know that (\ref{-1}) is true for $c<t$. We are left with the proof of (\ref{-1}) for $c\in\{t,\ldots ,d-1\}$. By Corollary \ref{equivcond} the congruence (\ref{-1}) is equivalent to the following one:
\begin{align}\label{-1m}
(-1)^c\sum_{k=1}^{\left\lfloor\frac{c}{t}\right\rfloor}\sum_{-c\leq i_1<\ldots <i_{kt}\leq -1} \frac{d^{kt}}{i_1\cdot \ldots \cdot i_{kt}}\equiv 0\pmod{p^{\nu_p(d)}}.
\end{align}
A summand of the left hand side has the least $p$-adic valuation if the set $\{i_1,\ldots ,i_{kt}\}$ contains all the values with $p$-adic valuation greater than $\nu_p(d)$ and as few values with $p$-adic valuation less than $\nu_p(d)$ as possible. Hence, the value of $k$ for this summand is the least positive integer such that $kt\geq\lfloor\frac{d-1}{p^{\nu_p(d)+1}}\rfloor$. This means that $$k=k_p(d)=\max\left\{1,\left\lceil\frac{1}{t}\left\lfloor\frac{d-1}{p^{\nu_p(d)+1}}\right\rfloor\right\rceil\right\}.$$ Since $\{i_1,\ldots ,i_{kt}\}\subset\{1-d,\ldots ,-1\}$, the minimal $p$-adic valuation is attained, when $\{i_1,\ldots ,i_{kt}\}$ contains $\left\lfloor\frac{d-1}{p^{\gamma_p(d)}}\right\rfloor$ values with $p$-adic valuation $\gamma_p(d)$, $\left\lfloor\frac{d-1}{p^{\gamma_p(d)-h}}\right\rfloor-\left\lfloor\frac{d-1}{p^{\gamma_p(d)-h+1}}\right\rfloor$ values with $p$-adic valuation $\gamma_p(d)-h$ for each $h\in\{1,\ldots ,\alpha_p\}$ and $k_p(d)t-\left\lfloor\frac{d-1}{p^{\gamma_p(d)-\alpha_p}}\right\rfloor$ values with $p$-adic valuation $\gamma_p(d)-\alpha_p-1$, where $\alpha_p$ is the greatest integer such that $k_p(d)t-\left\lfloor\frac{d-1}{p^{\gamma_p(d)-\alpha_p}}\right\rfloor >0$. Thus, $$\alpha_p=\left\lfloor\log_p \frac{k_p(d)tp^{\gamma_p(d)}}{d}\right\rfloor.$$ After the above preparation we estimate the least possible $p$-adic valuation of $\frac{d^{k_p(d)t}}{i_1\cdot \ldots \cdot i_{k_p(d)t}}$:
\begin{align*}
& \nu_p\left(\frac{d^{k_p(d)t}}{i_1\cdot \ldots \cdot i_{k_p(d)t}}\right)\geq \nu_p(d)k_p(d)t-\left\lfloor\frac{d-1}{p^{\gamma_p(d)}}\right\rfloor\gamma_p(d)\\
& -\sum_{h=1}^{\alpha_p}\left(\left\lfloor\frac{d-1}{p^{\gamma_p(d)-h}}\right\rfloor-\left\lfloor\frac{d-1}{p^{\gamma_p(d)-h+1}}\right\rfloor\right)(\gamma_p-h)\\
& -\left(k_p(d)t-\left\lfloor\frac{d-1}{p^{\gamma_p(d)-\alpha_p}}\right\rfloor\right)(\gamma_p-\alpha_p-1)\\
& =(\nu_p(d)-\gamma_p(d))k_p(d)t+\sum_{h=1}^{\alpha_p}\left(\left\lfloor\frac{d-1}{p^{\gamma_p(d)-h}}\right\rfloor-\left\lfloor\frac{d-1}{p^{\gamma_p(d)-h+1}}\right\rfloor\right)h\\
& +\left(k_p(d)t-\left\lfloor\frac{d-1}{p^{\gamma_p(d)-\alpha_p}}\right\rfloor\right)(\alpha_p+1)\\
& =(\nu_p(d)-\gamma_p(d))k_p(d)t\\
& +\sum_{g=1}^{\alpha_p}\left(\sum_{h=g}^{\alpha_p}\left(\left\lfloor\frac{d-1}{p^{\gamma_p(d)-h}}\right\rfloor-\left\lfloor\frac{d-1}{p^{\gamma_p(d)-h+1}}\right\rfloor\right)+\left(k_p(d)t-\left\lfloor\frac{d-1}{p^{\gamma_p(d)-\alpha_p}}\right\rfloor\right)\right)\\
& +\left(k_p(d)t-\left\lfloor\frac{d-1}{p^{\gamma_p(d)-\alpha_p}}\right\rfloor\right)\\
& =(\nu_p(d)-\gamma_p(d))k_p(d)t+\sum_{g=1}^{\alpha_p+1}\left(k_p(d)t-\left\lfloor\frac{d-1}{p^{\gamma_p(d)-g+1}}\right\rfloor\right)\\
& =(\nu_p(d)-\gamma_p(d)+\alpha_p+1)k_p(d)t-\sum_{g=0}^{\alpha_p}\left\lfloor\frac{d-1}{p^{\gamma_p(d)-g}}\right\rfloor\geq \nu_p(d).
\end{align*}
We have now proved (\ref{-1m}) for any $c\in\{t,\ldots ,d-1\}$ and we are done.
\end{proof}

An easy consequence of Proposition \ref{suffcond} is the following result for $d$ being a power of a prime number.

\begin{corollary}\label{power}
Let $d=p^r$ for some prime number $p$ and positive integer $r\leq (\alpha_p+1)(t+1)-\frac{p^{\alpha_p+1}-1}{p-1}$, where $\alpha_p=\lfloor\log_p t\rfloor$. Then $d$ satisfies the condition $C_t$.
\end{corollary}

\begin{proof}
We see that for $d=p^r$ we have $$k_p(p^r)=1,\ \alpha_p=\left\lfloor\log_p \frac{tp^{\gamma_p(p^r)}}{p^r}\right\rfloor=\left\lfloor\log_p t\right\rfloor\ \text{and}\ \left\lfloor\frac{p^r-1}{p^{\gamma_p(p^r)-h}}\right\rfloor =\left\lfloor\frac{p^r-1}{p^{r-h}}\right\rfloor =p^h-1.$$ Then the lower bound for the $p$-adic valuation of $\frac{p^{rt}}{i_1\cdot \ldots \cdot i_{kt}}$ takes the form 
\begin{align*}
& (\nu_p(p^r)-\gamma_p(p^r)+\alpha_p+1)t-\sum_{g=0}^{\alpha_p}\left\lfloor\frac{p^r-1}{p^{\gamma_p(p^r)-g}}\right\rfloor =(\alpha_p+1)t-\sum_{g=0}^{\alpha_p}(p^g-1)\\
& =(\alpha_p+1)(t+1)-\frac{p^{\alpha_p+1}-1}{p-1}.
\end{align*}
Hence $d$ satisfies the condition $C_t$, which was to be proved.
\end{proof}

Summing up the results from this section, we give the proof of Theorem \ref{main}.

\begin{proof}[Proof of Theorem \ref{main}]
If $d\leq t$, then $d$ satisfies the condition $C_t$ by Corollary \ref{equal}. 

If $d=p^r$, where $p$ is a prime number and $r$ is a positive integer not from the interval $$\left((\alpha_p+1)(t+1)-\frac{p^{\alpha_p+1}-1}{p-1},(\alpha_p+1)t-\nu_p(t!)\right],$$ then we check the validity of the condition $C_t$ by Corollary \ref{power} for $r\leq (\alpha_p+1)(t+1)-\frac{p^{\alpha_p+1}-1}{p-1}$ and failure of $C_t$ by Proposition \ref{toolargevaluation} for $r>(\alpha_p+1)t-\nu_p(t!)$. If additionally $p$ is an odd prime number and $t=p^u$ for some positive integer $u$, then from Corollary \ref{powernot} we conclude the failure of the condition $C_t$ for $r>(\alpha_p+1)(t+1)-\frac{p^{\alpha_p+1}-1}{p-1}$. Indeed, Corollary \ref{powernot} states the failure of $C_t$ for $r>ut+1-\nu_p((t-1)!)$. Thus, it suffices to show that $$ut+1-\nu_p((t-1)!)=(\alpha_p+1)(t+1)-\frac{p^{\alpha_p+1}-1}{p-1}.$$ First, we see that $\alpha_p=\lfloor\log_pt\rfloor =u$. Second, we compute
$$\nu_p((t-1)!)=\sum_{j=1}^{u}\left\lfloor\frac{t-1}{p^j}\right\rfloor =\sum_{j=1}^{u-1}(p^{u-j}-1)=\left(\sum_{j=0}^{u-1}p^j\right)-u=\frac{p^u-1}{p-1}-u.$$
Finally, compute
\begin{align*}
& ut+1-\nu_p((t-1)!)=ut+1-\frac{p^u-1}{p-1}+u=ut+u+t+1-\frac{p^u-1}{p-1}-t\\
& =(u+1)(t+1)-\frac{p^u-1}{p-1}-p^u=(u+1)(t+1)-\frac{p^{u+1}-1}{p-1}\\
& =(\alpha_p+1)(t+1)-\frac{p^{\alpha_p+1}-1}{p-1}.
\end{align*}

If $d$ is a composite number not being a power of prime number and such that $d>tp^{\nu_p(d)}$, then $d$ does not satisfy the condition $C_t$ by Proposition \ref{toolargevalue}.
\end{proof}

\section{Remarks and examples}

If we fix $t\in\N_+$, Theorem \ref{main} gives us a criterion for satisfying $C_t$ for all but finitely many positive integers $d$. The theorem does not cover the case of $d\in A_t$ but then we can try to apply Propositions \ref{beta}, \ref{toolargevaluation}, \ref{particular}, \ref{suffcond} or the following one. 

\begin{prop}\label{another}
Let $t=qp^u$, $d=(q+1)p^r$, where $p$ is a prime number and $q,u,r$ are positive integers such that $p\nmid q(q+1)$ and $r>ut-\nu_p(t!)$. Then $d$ does not satisfy the condition $C_t$.
\end{prop}

\begin{proof}
By Lemma \ref{equivconds} and Corollary \ref{equivcond}, it suffices to show that (\ref{cong}) does not hold for some $c<d$ and $n\in\Z$. We take $c=qp^r$ and $n=-1$ in (\ref{cong}) and check the validity of the following congruence:
\begin{equation}\label{c4}
\begin{split}
& \frac{1}{(qp^r)!}\sum_{k=1}^{p^{r-u}}(q+1)^{kt}p^{rkt}\sum_{-qp^r\leq i_1<\ldots <i_{qp^r-kt}\leq -1} j_1\cdot \ldots \cdot j_{qp^r-kt}\\
& =(-1)^{qp^r}\sum_{k=1}^{p^{r-u}}\sum_{-qp^r\leq i_1<\ldots <i_{kt}\leq -1} \frac{(q+1)^{kt}p^{rkt}}{i_1\cdot \ldots \cdot i_{kt}}\equiv 0\pmod{(q+1)p^r}.
\end{split}
\end{equation}
The summand on the left hand side with the least $p$-adic valuation is equal to
$$\frac{(q+1)^{t}p^{rt}}{\prod_{l=1}^t (l-t-1)p^{r-u}},$$
obtained for $k=1$ and $i_l=(l-t-1)p^{r-u}$, where $l\in\{1,\ldots ,t\}$. Its $p$-adic valuation is equal to $tu-\nu_p(t!)$. Any other summand has greater $p$-adic valuation since any value other than $(l-t-1)p^{r-u}$, $l\in\{1,\ldots ,t\}$, has $p$-adic valuation smaller than $r-u$. Thus, the left hand side of (\ref{c4}) has $p$-adic valuation $tu-\nu_p(t!)<r$. This means that the congruence \eqref{c4} is not true.
\end{proof}

However, it is possible that there still remain some values $d$ to be checked. We have no criterion for them. The reason is that we cannot compute a $p$-adic valuation of any expression 
$$\frac{1}{c!}\sum_{k=1}^{\left\lfloor\frac{c}{t}\right\rfloor}d^{kt}\sum_{n-c+1\leq j_1<\ldots <j_{c-kt}\leq n} j_1\cdot \ldots \cdot j_{c-kt}$$
 by finding a unique summand with the minimal $p$-adic valuation, for some prime divisor $p$ of $d$. Then, by Lemma \ref{equivconds} and Corollary \ref{equal}, having some remaining value $d$, it suffices to test the validity of congruences (\ref{gid}) for one fixed value $n\in\Z$ and all the values $c\in\{t,\ldots ,d-1\}$.

We support the above discussion by giving all the pairs $(t,d)\in\{1,2,3,4,5\}\times\N_+$ such that $d$ satisfies the condition $C_t$.

\bigskip

If $t=1$, then the condition $C_t=C_1$ takes the form
\begin{align*}
{n+d\choose d-1}\equiv {n\choose d-1}\pmod{d}
\end{align*}
for each integer $n$. If $d>1$, then Corollary \ref{simple} is sufficient to claim that $d$ satisfies $C_1$ if and only if $d$ is a prime number.

\bigskip

If $t=2$, then the condition $C_t=C_2$ takes the form
\begin{align*}
\frac{1}{2}\left({n+d\choose d-1}+{n-d\choose d-1}\right)\equiv {n\choose d-1}\pmod{d}
\end{align*}
for each integer $n$. This is, up to the factor $\frac{1}{2}$, the condition from the proof of \cite[Theorem 8.10]{MiUl}. Theorem \ref{main} suffices to claim that $C_2$ is satisfied if and only if $d\in\{1,8\}$ or $d=p^r$, where $p$ is a prime number and $r\in\{1,2\}$.

\bigskip

If $t=3$, then the condition $C_t=C_3$ takes the form
\begin{align*}
\frac{1}{3}\left({n+d\choose d-1}+{n+d\frac{i\sqrt{3}-1}{2}\choose d-1}+{n+d\frac{-i\sqrt{3}-1}{2}\choose d-1}\right)\equiv {n\choose d-1}\pmod{d}
\end{align*}
for each integer $n$, where $i$ is the imaginary unit. By Lemma \ref{equivconds} and Corollary \ref{equal}, its equivalent version is as follows (here we put $n=-1$):
\begin{align}\label{t=3,n=-1}
\frac{1}{c!}\sum_{k=1}^{\left\lfloor\frac{c}{3}\right\rfloor}\sum_{-c\leq i_1<\ldots <i_{3k}\leq -1} \frac{d^{3k}(-1)_c}{i_1\cdot \ldots \cdot i_{3k}}\equiv 0\pmod{d}, \quad c\in\{t,\ldots ,d-1\}.
\end{align}
Theorem \ref{main} implies that if $d\neq 6$, then $C_3$ is satisfied if and only if $d\in\{1,16,32,81\}$ or $d=p^r$, where $p$ is a prime number and $r\in\{1,2,3\}$. The value $d=6$ cannot be verified by any of Propositions \ref{beta}, \ref{toolargevaluation}, \ref{particular}, \ref{suffcond} or \ref{another}. Hence, we check the congruence \eqref{t=3,n=-1} for $d=6$:
\begin{align*}
(-1)^c\sum_{-c\leq i_1<\ldots <i_{3k}\leq -1} \frac{6^{3k}}{i_1\cdot \ldots \cdot i_{3k}}\equiv 0\pmod{6},\quad c\in\{3,4,5\}.
\end{align*}
We see that the above congruence is false for $c=5$ as we have three summands with $2$-adic valuation $0$, namely
$$\frac{6^3}{(-5)\cdot (-4)\cdot (-2)},\ \frac{6^3}{(-4)\cdot (-3)\cdot (-2)},\ \frac{6^3}{(-4)\cdot (-2)\cdot (-1)},$$
and the remaining ones have positive $2$-adic valuation. Thus, the condition $C_3$ is not satisfied by $d=6$.

\bigskip

If $t=4$, then the condition $C_t=C_4$ takes the form
\begin{align*}
\frac{1}{4}\left({n+d\choose d-1}+{n+di\choose d-1}+{n-d\choose d-1}+{n-di\choose d-1}\right)\equiv {n\choose d-1}\pmod{d}
\end{align*}
for each integer $n$. Equivalently, using Lemma \ref{equivconds} and Corollary \ref{equal}, and putting $n=-1$, we have
\begin{align*}
\frac{1}{c!}\sum_{k=1}^{\left\lfloor\frac{c}{4}\right\rfloor}\sum_{-c\leq i_1<\ldots <i_{4k}\leq -1} \frac{d^{4k}(-1)_c}{i_1\cdot \ldots \cdot i_{4k}}\equiv 0\pmod{d}.
\end{align*}
Theorem \ref{main} implies that if $d\not\in\{6,12,512,2187\}$, then $C_4$ is satisfied if and only if $d\in\{1, 32, 64, 128, 243, 256, 729\}$ or $d=p^r$, where $p$ is a prime number and $r\in\{1,2,3,4\}$. For $d=6$ we apply Proposition \ref{suffcond} to claim that $C_4$ holds. For $d=12$ we apply Proposition \ref{beta} with $p=2$ and $\beta_2=1$ to see that $C_4$ does not hold. For $d\in\{512,2187\}$ we should check the congruences
\begin{align}\label{t=4}
(-1)^c\sum_{k=1}^{\left\lfloor\frac{c}{4}\right\rfloor}\sum_{-c\leq i_1<\ldots <i_{4k}\leq -1} \frac{2187^{4k}}{i_1\cdot \ldots \cdot i_{4k}}\equiv 0\pmod{2187}
\end{align}
for $c\in\{4,\ldots ,d-1\}$. 

For $d=512$ we consider $c=384$. There are three summands with the least $2$-adic valuation, equal to $8$: $\frac{512^4}{256\cdot 128\cdot 384\cdot 64}$, $\frac{512^4}{256\cdot 128\cdot 384\cdot 192}$ and $\frac{512^4}{256\cdot 128\cdot 384\cdot 320}$. The remaining summands in the left hand side of (\ref{t=4}) have $2$-adic valuation greater than $8$, which means that the $2$-adic valuation of (\ref{t=4}) is $8$. Hence, $C_4$ is not satisfied for $d=512$.

For $d=2187$ we consider $c=1701$. The summands with the least $3$-adic valuation, equal to $6$, are $\frac{2187^4}{729\cdot 1458\cdot 243a\cdot 243b}$, where $a,b\in\{1,2,4,5,7\}$ and $a<b$. Hence, their sum is
$$\sum_{1\leq a<b\leq 7, 3\nmid ab}\frac{2187^4}{729\cdot 1458\cdot 243a\cdot 243b}=\frac{3^6}{2}\sum_{1\leq a<b\leq 7, 3\nmid ab}\frac{1}{ab}.$$
Since $\sum_{1\leq a<b\leq 7, 3\nmid ab}\frac{1}{ab}\equiv 1\pmod{3}$, the $3$-adic valuation of the left hand side in the above identity is $6$. The remaining summands in the left hand side of (\ref{t=4}) have $3$-adic valuation greater than $6$, which means that (\ref{t=4}) is not satisfied. Hence, $C_4$ does not hold for $d=2187$.

\bigskip

If $t=5$, then the condition $C_t=C_5$ takes the form
\begin{align*}
\frac{1}{5}\sum_{s=0}^4{n+d\zeta_5^s\choose d-1}\equiv {n\choose d-1}\pmod{d}
\end{align*}
for each integer $n$. Equivalently, using Lemma \ref{equivconds} and Corollary \ref{equal}, and putting $n=-1$, we have
\begin{align*}
\frac{1}{c!}\sum_{k=1}^{\left\lfloor\frac{c}{5}\right\rfloor}\sum_{-c\leq i_1<\ldots <i_{5k}\leq -1} \frac{d^{5k}(-1)_{d-1}}{i_1\cdot \ldots \cdot i_{5k}}\equiv 0\pmod{d}, \quad c\in\{5,\ldots ,d-1\}.
\end{align*}
Theorem \ref{main} implies that if $d\not\in\{6,12,15,20,4096,19683\}$, then $C_5$ is satisfied if and only if $d\in\{1,64,128,256,512,729,1024,2048,2187,6561,15625\}$ or $d=p^r$, where $p$ is a prime number and $r\in\{1,2,3,4,5\}$. For $d\in\{6,12\}$ we apply Proposition \ref{suffcond} to conclude that $C_5$ holds. For $d\in\{15,20,4096,19683\}$ we should check the congruences
\begin{align}\label{t=5}
(-1)^c\sum_{k=1}^{\left\lfloor\frac{c}{5}\right\rfloor}\sum_{-c\leq i_1<\ldots <i_{5k}\leq -1} \frac{d^{5k}}{i_1\cdot \ldots \cdot i_{5k}}\equiv 0\pmod{p^{\nu_p(d)}}
\end{align}
for $c\in\{5,\ldots ,d-1\}$ and any prime divisor $p$ of $d$.

For $d=15$ we take $c=13$ and see that the summands in the above congruence with the least $3$-adic valuation, equal to $0$, are of the form $\frac{15^5}{3\cdot 6\cdot 9\cdot 12\cdot a}$, where $a\in\{1,\ldots ,13\}$ and $3\nmid a$. Then the sum of these summands is
$$\sum_{a=1, 3\nmid a}^{13}\frac{15^5}{3\cdot 6\cdot 9\cdot 12\cdot a}=\frac{5^5}{8}\sum_{a=1, 3\nmid a}^{13}\frac{1}{a}\equiv 1\pmod{3}$$
and thus it has $3$-adic valuation equal to $0$. As a result, the left hand side of congruence (\ref{t=5}) has $3$-adic valuation equal to $0$, which means that the condition $C_5$ does not hold.

For $d=20$ we take $c=18$ and see that there are $5$ summands in the above congruence with the least $2$-adic valuation.  They are of the form
$$\frac{20^5}{4\cdot 8\cdot 12\cdot 16\cdot 2a},\ a\in\{1,3,5,7,9\}$$
and their $2$-adic valuation is equal to $-2$. As a result, the left hand side of congruence (\ref{t=5}) has $2$-adic valuation equal to $-2$, which means that the condition $C_5$ does not hold.

For $d=4096=2^{12}$ we take $c=3072$ and see that there are $3$ summands in the above congruence with the least $2$-adic valuation. They are of the form
$$\frac{4096^5}{1024\cdot 2048\cdot 3072\cdot 512a\cdot 512b},\ a,b\in\{1,3,5\},\ a<b$$
and their $2$-adic valuation is equal to $11$. As a result, the left hand side of congruence (\ref{t=5}) has $2$-adic valuation equal to $11$, which means that the condition $C_5$ does not hold.

For $d=19683=3^9$ we take $c=15309$ and see that the summands in the above congruence with the least $3$-adic valuation, equal to $8$, are of the form
$$\frac{19683^5}{6561\cdot 13122\cdot 2187a\cdot 2187b\cdot 2187c},\ a,b,c\in\{1,2,4,5,7\},\ a<b<c.$$ Then the sum of these summands is
\begin{align*}
& \sum_{1\leq a<b<c\leq 7, 3\nmid abc}\frac{19683^5}{6561\cdot 13122\cdot 2187a\cdot 2187b\cdot 2187c}\\
& =\frac{3^8}{2\cdot 1\cdot 2\cdot 4\cdot 5\cdot 7}\sum_{1\leq x<y<z\leq 7, 3\nmid xyz}xyz\equiv 2\cdot 3^8\pmod{3^9}
\end{align*}
and thus it has $3$-adic valuation equal to $8$. As a result, the left hand side of congruence (\ref{t=5}) has $3$-adic valuation equal to $8$, which means that the condition $C_5$ does not hold.

\section*{Acknowledgements}

The author wishes to thank the anonymous referee for their careful reading of the paper and remarks that improved its edition.


\begin{thebibliography}{99}
\bibitem{Bai} D. F. Bailey, \textit{Two $p^3$ variations of Lucas’ theorem}, J. Number Theory 35 (1990), 208--215.
\bibitem{Bru} J. W. Bruce, \textit{A really trivial proof of the Lucas-Lehmer test}, Amer. Math. Monthly 100 (1993), 370--371.
\bibitem{ChamDil} M. Chamberland, K. Dilcher, \textit{A binomial sum related to Wolstenholme's theorem}, J. Number Theory 129(2009), 2659--2672.
\bibitem{CaoSun} H. Q. Cao, Z. W. Sun, \textit{Some congruences involving binomial coefficients}, Colloq. Math. 139 (1) (2015), 127--136.
\bibitem{DavWeb1} K. S. Davis, W. A. Webb, \textit{Lucas’ theorem for prime powers}, European  J. Combin. 11 (1990), 229--233.
\bibitem{DavWeb2} K. S. Davis, W. A. Webb, \textit{A binomial coefficient congruence modulo prime powers}, J. Number Theory 43 (1993), 20--23.
\bibitem{Gla} J. W. L. Glaisher, \textit{Congruences relating to the sums of products of the first nnumbers and to the other sums of products}, Quart. J. Math. 31 (1900), 1--35.
\bibitem{Gran}  A. Granville, \textit{Arithmetic properties of binomial coefficients. I. Binomial coefficients modulo prime powers}, in: Organic mathematics (Burnady, BC, 1995), 253–276, CMS Conf. Proc., 20, Amer. Math. Soc., Providence, RI, 1997.
\bibitem{HelTer}  C. Helou, G. Terjanian, \textit{On Wolstenholme’s theorem and its converse}, J. Number Theory 128 (2008), 475--499.
\bibitem{McIn} R. J. McIntosh, \textit{On the converse of Wolstenholme's Theorem}, Acta Arith. 71 (1995), 381--389.
\bibitem{McInRoe} R. J. McIntosh, E. L. Roettger, \textit{A search for Fibonacci-Wieferich and Wolstenholme primes}, Math. Comp. 76 (2007), 2087--2094.
\bibitem{MengSun} X. Z. Meng, Z. W. Sun, \textit{Proof of a conjectural supercongruence}, Finite Fields Appl. 35 (2015), 86--91.
\bibitem{Mes} R. Me\v{s}trovi\'{c}, \textit{A note on the congruence ${np^k\choose mp^k}\equiv {n\choose m}\pmod{p^r}$}, Czechoslovak Math. J. 62 (1) (2012), 59--65.
\bibitem{MiUl} P. Miska, M. Ulas, \textit{Arithmetic properties of the number of permutations being products of pairwise disjoint $d$-cycles},  Monatsh. Math. (2020), 59 pp., on-line first 18.03.2020, DOI: 10.1007/s00605-020-01397-5.
\bibitem{Rod} {\O}. J. R{\o}dseth, \textit{A note on primality tests for $N=h\cdot 2^n-1$}, BIT Numerical Mathematics, 34 (3) (1994), 451--454.
\bibitem{Sun1} Z. W. Sun, \textit{A congruence for primes}, Proc. Amer. Math. Soc. 123 (1995), 1341–-1346.
\bibitem{Sun2} Z. W. Sun, \textit{On the sum $\sum_{k\equiv r\pmod{m}}{n\choose k}$ and related congruences}, Israel J. Math., 128 (2002), 135--156.
\bibitem{Sun3} Z. W. Sun, \textit{Polynomial extension of Fleck's congruence}, Acta Arith., 122 (1) (2006), 91--100.
\bibitem{SunDav} Z. W. Sun, D. Davis, \textit{Combinatorial congruences modulo prime powers}, Trans. Amer. Math. Soc., 359 (11) (2007), 5525--5553.
\bibitem{SunTau1} Z. W. Sun, R. Tauraso, \textit{Congruences for sums of binomial coefficients}, J. Number Theory 126 (2) (2007), 287--296.
\bibitem{SunTau2} Z. W. Sun, R. Tauraso, \textit{On some new congruences for binomial coefficients}, Int. J. Number Theory 7 (3) (2011), 645--662.
\bibitem{Sun4} Z. W. Sun, \textit{On sums of binomial coefficients modulo $p^2$}, Colloq. Math. 127 (2012), 39--54.
\bibitem{Sun5} Z. W. Sun, \textit{Congruences involving generalized central binomial coefficients}, Sci. China Math. 57 (2014), 1375--1400.
\bibitem{Wan} D. Wan, \textit{Combinatorial congruences and $\psi$-operators}, Finite Fields Appl., 12 (2006), 693--703.
\bibitem{Zhao} J. Zhao, \textit{Bernoulli numbers, Wolstenholme's theorem, and $p^5$ variations of Lucas' theorem}, J. Number Theory 123 (2007), 18--26.
\bibitem{ZhaoSun} L. L. Zhao, Z. W. Sun, \textit{Some curious congruences modulo primes}, J. Number Theory 130 (4) (2010), 930--935.
\end{thebibliography}
\end{document}